\documentclass[11pt, amsfonts]{amsart}

\usepackage{amsmath,amssymb,amsthm,color}
\usepackage[all]{xy}

\SelectTips{eu}{12}

\newtheorem{thm}{Theorem}

\newtheorem{lem}[thm]{Lemma}
\newtheorem{cor}[thm]{Corollary}

\theoremstyle{definition}

\theoremstyle{remark}

\newcommand{\bt}{{\, \bowtie}}
\newcommand{\FF}{\mathbb{F}}

\newcommand{\DD}{\mathcal{D}}
\newcommand{\HH}{\mathcal{H}}

\newcommand{\ZZ}{\mathbb{Z}}
\newcommand{\conn}{{\rm conn}}
\newcommand{\coind}{{\rm coind}}
\newcommand{\hdim}{\textrm{\rm h-dim}}

\title{Dominance complexes and vertex cover numbers of graphs}

\author{Takahiro Matsushita}
\address{Department of Mathematical Sciences, University of the Ryukyus, Nishihara-cho, Okinawa 903-0213, Japan}
\email{mtst@sci.u-ryukyu.ac.jp}



\begin{document}

\baselineskip.525cm

\maketitle

\begin{abstract}
The dominance complex $D(G)$ of a simple graph $G = (V,E)$ is the simplicial complex consisting of the subsets of $V$ whose complements are dominating. We show that the connectivity of $D(G)$ plus $2$ is a lower bound for the vertex cover number $\tau(G)$ of $G$.
\end{abstract}


\section{Introduction}

For a simple graph $G = (V,E)$, a subset $S$ of $V$ is {\it dominating in $G$} if every vertex $v$ in $G$ is contained in $S$ or adjacent to an element in $S$. The {\it dominance complex $D(G)$} is the simplicial complex consisting of the subsets of $V$ whose complements are dominating. Dominance complex was considered in Ehrenborg and Hetyei \cite{EH}, and has been studied by several authors (see \cite{MT1}, \cite{MT2}, and \cite{Taylan}).

The goal of this paper is to establish a certain relationship between the domincance complex $D(G)$ and the vertex cover number $\tau(G)$ of a graph $G$. Recall that a \emph{vertex cover of a simple graph $G =(V,E)$} is a subset $S$ of $V$ such that every edge of $G$ contains at least one element of $S$. The \emph{vertex cover number $\tau(G)$ of $G$} is the smallest cardinality of a possible vertex cover of $G$, and is one of the most classical invariants in graph theory.

Before stating our main result, we review concrete examples of dominance complexes whose homotopy types are determined. Ehrenborg and Hetyei \cite{EH} showed that the dominance complex of a forest is homotopy equivalent to a sphere, and Marietti and Testa \cite{MT1} in fact showed $S^{\tau(G) - 1} \simeq D(G)$ when $G$ is a forest. Taylan generalized this result by Marietti and Testa to chordal graphs (see Theorem 5.4 of \cite{Taylan}). She also determined the homotopy types of the $P_3$-devoid complexes of cycles, which coincide with the dominance complexes of cycles (see also Theorem 4.6 of \cite{DS}). More precisely, Taylan showed
$$D(C_{4t}) \simeq S^{2t-1} \vee S^{2t-1} \vee S^{2t-1}, \; D(C_{4t + i}) \simeq S^{2t + i - 2} \;\;\; (i = 1,2,3).$$
Note that the vertex cover number of $C_n$ is $\lceil n / 2 \rceil$.

These results seem to suggest that there is some relationship between the vertex cover number $\tau(G)$ of $G$ and the connectivity $\conn(D(G))$ of the dominance complex $D(G)$ of $G$. The goal of this note is to show that a certain homotopy invariant of $D(G)$ related to the connectivity provides a lower bound for $\tau(G)$.

Let $\ZZ_2$ denote the cyclic group of order $2$. For a topological space $X$, let $\conn_{\ZZ_2}(X)$ be the largest number $n$ such that $i \le n$ implies $\widetilde{H}_i (X ; \ZZ_2) = 0$, and call it the {\it $\ZZ_2$-homological connectivity of $X$}. Note that the Hurewicz theorem (Theorem 4.32 \cite{Hatcher}) and the universal coefficient theorem (Theorem 3A.4 of \cite{Hatcher}) imply the inequality $\conn(X) \le \conn_{\ZZ_2}(X)$. Then our main result is formulated as follows:

\begin{thm} \label{thm 2}
For every simple graph $G$, the following inequality holds:
$$\conn_{\ZZ_2} (D(G)) + 2 \le \tau(G).$$
\end{thm}

\begin{cor}
For every simple graph $G$, the dominance complex $D(G)$ is not contractible.
\end{cor}

Note that in the examples of graphs mentioned above, the equality $\conn_{\ZZ_2}(D(G)) + 2 = \tau(G)$ holds except for the case $G = C_{4t + 1}$. In this case, these numbers differ by $1$.

In the proof of Theorem \ref{thm 2}, we show that the suspension $\Sigma (D(G)^\vee)$ of the combinatorial Alexander dual of the dominance complex of a graph $G$ is homotopy equivalent to the independence complex $I(G^\bt)$ of a graph $G^\bt$ defined in the next section. In that proof, we see that the independence complex of a hypergraph provides a simple formulation of a result by Nagel and Reiner \cite{NR} concerning independence complexes of bipartite graphs, following Tsukuda \cite{Tsukuda}.

The graph $G^\bt$ has a natural involution, and $I(G^\bt)$ becomes a free $\ZZ_2$-complex. Recall that the coindex of a $\ZZ_2$-space $X$ is the largest integer $n$ such that there exists a $\ZZ_2$-map from $S^n$ to $X$, which has several interesting applications in combinatorics (see Matou\v{s}ek \cite{Matousek}). By the definition of $G^\bt$, it will be seen that a lower bound of the coindex of $I(G^\bt)$ is provided by $\alpha(G) = |V| - \tau (G)$, {\it i.e.,} the size of a maximum independent set of $G$ (Lemma \ref{lem 7}). This observation is a key to the proof of Theorem \ref{thm 2}.

\section{Proofs}

We first show that the suspension of the combinatorial Alexander dual of the dominance complex $D(G)$ is homotopy equivalent to an independence complex $I(G^\bt)$ of a certain graph $G^\bt$. To see this, we use a theorem by Nagel and Reiner (Theorem \ref{thm 2.1}), which states that for every simplicial complex $K$, there is a bipartite graph $G_K$ such that the independence complex $I(G_K)$ of $G_K$ is homotopy equivalent to the suspension $\Sigma K$ of $K$. We first see that by using the independence complexes of hypergraphs, we can simply describe the relationship between $K$ and $G_K$. Recall that the independence complex of a simple graph was introduced in \cite{BK}, and has been extensively studied (see \cite{Adamaszek}, \cite{Barmak}, \cite{Engstrom}, \cite{Jonsson}, \cite{Kawamura}, \cite{Kozlov}, and \cite{Meshulam}) in topological combinatorics. The independence complexes of hypergraphs are a natural generalization of it, and has often appeared in the study of simplicial complexes (see \cite{Dochtermann}, \cite{EH}, and \cite{Woodroofe} for example).

Now we recall some concepts related to independence complexes of hypergraphs. A {\it hypergraph $\HH = (X, H)$} is a pair consisting of a set $X$ equipped with a multi-set $H$ on $X$. We consider that every hypergraph is finite, {\it i.e.,} $X$ and $H$ are finite. A subset $\sigma$ of $X$ is {\it independent} if there is no element of $H$ contained in $\sigma$. Then the independent sets of $\HH$ form a simplicial complex $I(\HH)$, and we call it the {\it independence complex of $\HH$}. Note that every simplicial complex $K$ is isomorphic to some independence complex of a hypergraph. Indeed, if we define the hypergraph $\HH_K = (V(K), H)$ where $H$ is the set of non-faces of $K$, then $I(\HH_K)$ and $K$ are isomorphic.

Next we recall the combinatorial Alexander dual of a simplicial complex. Let $K$ be a simplicial complex with underlying set $X$. Then the {\it combinatorial Alexander dual $K^\vee$} is the simplicial complex consisting of the subsets of $X$ whose complement is a non-face of $K$. Then a simplex of the Alexander dual $I(\HH)^\vee$ of the independence complex of a hypergraph $\HH = (X,H)$ is a subset $\sigma$ of $X$ such that $\sigma \cap \tau = \emptyset$ for some $\tau \in H$. Recall that a subset of $X$ which intersects every hyperedge of $\HH$ is said to be {\it transversal}. Thus $I(\HH)^\vee$ is the simplicial complex consisting of non-transversal sets.


For a hypergraph $\HH = (X, H)$, let $B_\HH$ denote the incidence graph of the hypergraph $\HH$. Namely, the vertex set of $B_\HH$ is the disjoint union $X \sqcup H$ of $X$ and $H$, $X$ and $H$ are independent sets in $B_\HH$, and $v \in X$ and $h \in H$ are adjacent if and only if $v \in h$.

For a simplicial complex $K$, Nagel and Reiner constructed a graph $G_K$ such that $I(G_K)$ is homotopy equivalent to $\Sigma K$. Using the terminology of independence complexes of hypergraphs, their construction of $G_K$ is simply described as follows (see Corollary 4.11 of \cite{Tsukuda}):

\begin{thm}[Proposition 6.2 of \cite{NR}, see also Theorem 3.8 of \cite{Barmak} and Theorem 3.2 of \cite{Jonsson}] \label{thm 2.1}
For every hypergraph $\HH$, there is a following homotopy equivalence:
$$\Sigma \big(I(\HH)^\vee \big) \simeq I(B_\HH).$$
\end{thm}

We consider the case of dominance complex. As Ehrenborg and Hetyei noted in \cite{EH}, the dominance complex $D(G)$ of a simple graph $G$ is simply described as the independence complex of some hypergraph $\DD_G$ defined as follows: The underlying set of $\DD_G$ is the vertex set $V(G)$ of $G$, and the set of hyperedges of $\DD_G$ is the multi-set $\{ N[v] \; | \; v \in V(G)\}$. Here $N[v]$ denotes the set
$\{ v\} \cup \{ w \in V \; | \; \{ v,w\} \in E(G)\}$. Then it is easy to see $D(G) = I(\DD_G)$.

Next we describe the incidence graph of $\DD_G$. Define the graph $G^\bt$ as follows: The vertex set of $G^\bt$ is $\{ + , -\} \times V(G)$, and the set of edges of $G^\bt$ is
$$E(G^\bt) = \big\{ \{ (+, v), (-,w)\} \; | \; v \in N[w]\big\}.$$
Note that $(+,v)$ and $(-,v)$ are adjacent in $G^{\bt}$ for each vertex $v$ of $G$. Clearly, $G^\bt$ is isomorphic to the incidence graph of $\DD_G$, and Theorem \ref{thm 2.1} implies the following:

\begin{cor}
For every graph $G$, there is the following homotopy equivalence:
$$\Sigma \big( D(G)^\vee \big) \simeq I(G^\bt).$$
\end{cor}

Note that $G^\bt$ has a natural involution exchanging $(+,v)$ and $(-,v)$, and we write $\gamma$ to indicate the involution. Then $\gamma$ provides a $\ZZ_2$-action of $I(G^\bt)$.

\begin{lem}
For every graph $G$, the $\ZZ_2$-action of $I(G^\bt)$ is free.
\end{lem}
\begin{proof}
Let $\sigma$ be a simplex of $I(G^\bt)$. It suffices to show $\sigma \cap \gamma \sigma = \emptyset$. Suppose $\sigma \cap \gamma \sigma \ne \emptyset$ and let $(\varepsilon, v) \in \sigma \cap \gamma \sigma$. This means $(+,v), (-,v) \in \sigma$. Since $\sigma$ is an independent set in $G^\bt$, this is a contradiction.
\end{proof}

For a free $\ZZ_2$-space $X$, the {\it coindex $\coind(X)$ of $X$} is the largest integer $n$ such that there is a continuous $\ZZ_2$-map from $S^n$ to $X$. Here we consider the involution of $S^n$ as the antipodal map. Recall that $\alpha(G)$ denotes the size of a maximum independent set of a simple graph $G$.

\begin{lem} \label{lem 7}
Let $G$ be a graph. Then the complex $I(G^\bt)$ has a $\ZZ_2$-subcomplex which is $\ZZ_2$-homeomorphic to $S^{\alpha(G) - 1}$. In particular, the inequality $\alpha(G) - 1 \le \coind(I(G^\bt))$ holds.
\end{lem}
\begin{proof}
Let $A_n$ be the boundary of $(n+1)$-dimensional cross polytope. Namely, the vertex set of $A_n$ is $\{ \pm 1, \cdots, \pm (n+1)\}$ and a subset $\sigma$ of it is a simplex if and only if there is no $i$ satisfying $\{ \pm i\} \subset \sigma$. Then $|A_n|$ is homeomorphic to $S^n$.

Let $\sigma = \{ v_1, \cdots, v_{\alpha(G)}\}$ be a maximum independent set of $G$. Define the simplicial map $f \colon A_{\alpha(G) - 1} \to I(G^\bt)$ by sending $+i$ to $(+, v_i)$ and $-i$ to $(-, v_i)$. This is clearly an inclusion from $A_{\alpha(G) - 1}$ to $I(G^\bt)$ which is $\ZZ_2$-equivariant. This completes the proof.
\end{proof}

Next we observe that the coindex of a free $\ZZ_2$-space $X$ gives a restriction of the homology groups of $X$. Let $\hdim_{\ZZ_2}(X)$ be the maximum integer $n$ such that $\widetilde{H}_n(X ; \ZZ_2) \ne 0$. Then we have the following:

\begin{lem} \label{lem 6}
For a finite free $\ZZ_2$-simplicial complex $X$, the following inequality holds:
$$\coind(X) \le \hdim_{\ZZ_2}(X)$$
\end{lem}
\begin{proof}
Suppose that $n = \coind(X) > \hdim_{\FF_2}(X)$. Let $\overline{X}$ denote the orbit space of $X$ and $w_1(X)$ the $1$st Stiefel-Whitney class of the double cover $X \xrightarrow{p} \overline{X}$ (see \cite{Kozlov} or \cite{MS2}). Since there is a $\ZZ_2$-map $S^n \to X$ and $w_1(S^n)^n \ne 0$, the naturality of $w_1$ implies $0 \ne w_1(X)^n \in H^n(\overline{X} ; \ZZ_2)$. By the Gysin sequence for the double cover (see Corollary 12.3 of \cite{MS2}), we have the following exact sequence:
$$H^k(X ; \ZZ_2) \to H^k(\overline{X} ; \ZZ_2) \xrightarrow{\cup w_1(X)} H^{k+1}(\overline{X} ; \ZZ_2)$$
Since $\hdim_{\ZZ_2}(X) < n$, we have that $H^n(X ; \ZZ_2) = 0$ and hence the map  $H^n(\overline{X} ; \ZZ_2) \xrightarrow{\cup w_1(X)} H^{n+1}(\overline{X} ; \ZZ_2)$ is injective. Thus we have $w_1(X)^{n+1} \ne 0$. By induction, we have that $0 \ne w_1(X)^k \in H^k(\overline{X} ; \ZZ_2)$ for every $k > n$. This is a contradiction since $\overline{X}$ is a finite complex.
\end{proof}

We are now ready to complete the proof of Theorem \ref{thm 2}. Set $k = \conn_{\ZZ_2}(D(G))$. The combinatorial Alexander duality theorem (see \cite{MS1}) implies $\hdim_{\ZZ_2}(D(G)^\vee ) = |V| - k - 4$. Thus we have 
$$\alpha(G) - 1 \le \coind(I(G^\bt)) \le \hdim_{\ZZ_2}(I(G^\bt)) = \hdim_{\ZZ_2}\big( \Sigma (D(G)^\vee) \big) = |V| - k - 3.$$
Here the first and second inequalities follow from Lemma \ref{lem 6} and Lemma \ref{lem 7}, respectively. Thus we have
$$\conn_{\ZZ_2}\big( D(G) \big) + 2 = k + 2 \le |V| - \alpha(G) = \tau(G)$$
This completes the proof.

\section*{Acknowledgment}
The author is partially supported by JSPS KAKENHI 19K14536. The author thanks the referee for useful comments. The author states that there is no conflict of interests.

\end{document}